\documentclass[12pt,reqno]{amsart}
\usepackage{amsmath, amsfonts, amssymb, amsthm,mathrsfs,color,hyperref, url}
\usepackage{verbatim}
\usepackage{enumerate}
\usepackage{cite}
\usepackage{color}
\usepackage{colortbl,dcolumn}
\usepackage[margin=1in,footskip=0.25in]{geometry}

\def\Z{\mathbb Z}

\def\N{\mathbb N}

\def\R{\mathbb R}

\def\cB{\mathcal B}

\def\cA{{\mathcal A}}
\def\cB{{\mathcal B}}

\def\1{{\bf 1}}

\theoremstyle{plain}
\newtheorem{theorem}{Theorem}

\newtheorem{lemma}{Lemma}

\newtheorem{corollary}{Corollary}
\theoremstyle{definition}

\theoremstyle{remark}
\newtheorem*{Rem}{Remark}

\usepackage{todonotes}

\begin{document}

\author{Oleksiy Klurman}
\address{
Department of Mathematics,
University of Bristol, Woodland Rd, Bristol BS8 1UG}
\email{oleksiy.klurman@bristol.ac.uk}

\author{Marc Munsch}
\address{
Université Jean Monnet, Centrale Lyon, INSA Lyon, Universite Claude Bernard Lyon 1, CNRS, ICJ UMR5208, 42023 Saint-
Etienne, France.}
\email{marc.munsch@univ-st-etienne.fr}

\author{Yu-Chen Sun}
\address{
Department of Mathematics,
University of Bristol, Woodland Rd, Bristol BS8 1UG}
\email{yuchensun93@163.com}
\title[Small values of logarithmic means of multiplicative functions]
{Small values of signed harmonic sums and logarithmic means of multiplicative functions}

\maketitle
\begin{abstract} 
We construct sequences $\{a_n\}_{n\in\mathbb{N}}\in\{-1,1\}^{\mathbb{N}}$ with small values of signed harmonic sums
\[
\sum_{n\in\mathcal{A}\cap[1,N]}\frac{a_n}{n},
\]
for any reasonably dense subsets $\mathcal{A}\subset\mathbb{N}.$ We apply these methods to further construct completely multiplicative functions $f:\mathbb{N}\to\{-1,1\}$ with unusually small logarithmic partial sums, that is,
\[
\sum_{n \leq N}\frac{f(n)}{n} \ll \exp\left(-c_0 \frac{N^{1/3}}{(\log N)^{1/3}} \right)
\]
holds for infinitely many $N\to\infty$. The proofs combine careful analysis of the small-scale distribution of random harmonic sums over subsets of $\mathbb{N}$, together with deterministic inductive arguments inspired by the ``anatomy" of integers.
\end{abstract}
    \section{Introduction}
    \setcounter{lemma}{0} \setcounter{theorem}{0}
    \setcounter{equation}{0}
    Let $f$ be a completely multiplicative function that takes values in $\{-1,1\}$ and let
\[
\mathcal{M}(f,x):=\sum_{n \leq x}f(n), \quad \text{and} \quad L(f,x):=\sum_{n \leq x}\frac{f(n)}{n}.
\]
It is easy to show that for a ``typical" (random) $f,$ partial sums $\mathcal{M}(f,x)$ take both positive and negative values, and hence $\mathcal{M}(f,N)=0$ for infinitely many integers $N;$ for more precise results on sign changes of $\mathcal{M}(f,N),$ we refer the reader to some recent works \cite{Aymone2, Aymone, AHZ, GH, KLM}. If one seeks explicit examples of such functions, consider a ``modified" quadratic character, e.g. the completely multiplicative function such that
\[
\chi_{-3}^*(p)=
\begin{cases}
\left(\frac{p}{3}\right)  & (p,3)=1;\\
-1 &  p=3.
\end{cases}
\]
We have 
\[
\sum_{n \leq 3^{K}+1} \chi_{-3}^*(n) =1+ \sum_{k=0}^{K}(-1)^k\sum_{\substack{1 \leq a \leq 3^{K-k} \\ (a,3)=1}} \chi_{-3}(a)= (-1)^{K}+1,
\]
which implies that $\mathcal{M}(\chi_{-3},3^{2k+1})=0.$\\

In this paper, we confine ourselves to the much more subtle question of how small the logarithmic means $L(f,N)$ can be. Apart from being a natural question, the motivation for studying this problem is related to the infamous Siegel zero problem. Indeed, by Siegel's theorem, we have that for a quadratic character $\chi_d,$
$$L(1,\chi_d)=\sum_{n\ge 1}\frac{\chi_d(n)}{n}=\sum_{n\le d}\frac{\chi_d(n)}{n}+O\left(\frac{1}{\sqrt{d}}\right) > \frac{C(\epsilon)}{d^{\epsilon}},$$
and improving this lower bound constitutes a notorious open problem. Moreover, in a related work \cite{MR2024414}, Granville and Soundararajan showed that with positive probability there exists a family of Dirichlet characters $\chi_d,$ such that \[L(1,\chi_d)\ll \frac{1}{\log\log d}.\] 
It is therefore desirable to understand to what extent multiplicativity alone allows one to lower bound $L(f,N).$ 
Our first result shows that $L(f,N)$ can be remarkably small infinitely often.

\begin{theorem}\label{t:small_log_mean_mult_f}
There exist a completely multiplicative function $f:\mathbb{N}\to \{-1,1\}$, an infinite sequence $\{N_i\}_{i=0}^{\infty}$, and a constant $c_0 \in (0,1)$ such that for all $i> 0$,
\[
\sum_{1 \leq n \leq N_i} \frac{f(n)}{n} \ll \exp\left(-c_0 \frac{N_i^{1/3}}{(\log N_i)^{1/3}} \right)
.\]
\end{theorem}

To gauge the strength of Theorem \ref{t:small_log_mean_mult_f}, we make several preliminary remarks. First, if $f$ is a random completely multiplicative function (e.g. $f(p)$ is chosen to be $\pm 1$ independently with probability $1/2$), one can show that, with high probability, $L(f,N)$ can be well approximated by the Euler product
\[\prod_{p\le N}\left(1-\frac{f(p)}{p}\right)^{-1}\sim \exp\left(\sum_{p\le N}\frac{f(p)}{p}\right),\]
This implies that $L(f,N)$ is strongly concentrated around some constant value $c>0,$ which is obviously far from our target (unlike in the case of small values of unweighted partial sums $\mathcal{M}(f,N)$).

We further note that if $\lambda(n)$ is the Liouville function, then the prime number theorem implies that for any $A>0$,
\begin{equation}\label{e:Liouvill_pnt}
\sum_{n \leq x}\frac{\lambda(n)}{n} = o\left(\frac{1}{\log^A x} \right).
\end{equation}
On the other hand, assuming the Riemann Hypothesis, the above bound can be upgraded to  $o(x^{-1/2+\epsilon})$. It is also easy to see that the prime number theorem yields
\[
|L(f,N)|\geq \frac{1}{\text{lcm}(1,\dots,N)} = \exp(-(1+o(1))N),
\]
implying that the logarithmic sum can never decay faster than $\exp(-(1+o(1))N)$. Moreover, one immediately observes that $$L(f,N)\gg \frac{1}{N},$$ for infinitely many $N.$ These observations suggest that a reasonable goal is to look for an infinite subsequence $\{N_i\}_{i=0}^{\infty}$ such that $L(f,N_i)$ exhibits close to exponential decay, which is in line with the conclusion of our Theorem \ref{t:small_log_mean_mult_f}.\\

\subsection{Small values of signed harmonic sums.} 

In order to prove Theorem \ref{t:small_log_mean_mult_f}, we study the more general question of how small one can make 
\[{\bf X}_{\cA}=\sum_{n\in\mathcal{A}\cap[1,N]}\frac{a_n}{n},\]
where $a_i\in \{-1,1\}$ and $\mathcal{A}\subset \mathbb{N}$ is a subset of integers.\\ 
In the case of the full sum $\mathcal{A}=\mathbb{N},$ Bettin, Molteni, and Sanna \cite{BMS18} used probabilistic tools (see also \cite{BMS20} for a different ``greedy" approach) to verify the existence of a sequence $\{a_n\}_{n=1}^{N}$ in $\{-1,1\}^N$ such that 
\begin{equation}\label{e:har_ser_weak_bdd}
\sum_{n \leq N} \frac{a_n}{n} \ll \exp(- \log^2 N).
\end{equation}
This bound was later improved by Gambini, Tonon and Zaccagnini \footnote{This result is credited in their paper to J. Benatar and A. Nishry.} \cite[Corollary 15]{GTZ20} to $\exp(- N^{1/3-\epsilon}).$ To this end, we prove the following general result.

\begin{theorem}\label{t:postive_density_harmonic_series}
Let $\cA_0 \subset \N$ have positive lower density, that is, for some $0<\delta<1$,
\[
\liminf_{N \to \infty}\frac{|\cA_0 \cap [N]|}{N}>\delta.
\]
Then there exists a sufficiently large integer $N_0$ such that if $N>N_0$, there exists a sequence $\{a_n\}_{n=1}^{N}$ in $\{-1,1\}^N$ such that 
\[
\left|\sum_{n \in \cA_0 \cap[N]} \frac{a_n}{n} \right|\leq \exp\left(-N^{\theta(\delta)} \right), \quad \text{with } \theta(\delta)>0.
\]
\end{theorem}

\begin{Rem}
    The exponent $\theta(\delta)$ depends on the ``anatomy" of integers inside the set $\cA_0$. Moreover, for $\delta \to 0,$ the estimate $\theta(\delta) \gg (\log(1/\delta))^{-1}$ holds.
\end{Rem}

Interestingly, Theorem \ref{t:postive_density_harmonic_series} is optimal in several different ways. First, if we replace the infinite set $\mathcal{A}_0$ by a large finite subset of $[N]:=\{1,2,3,\dots,N\}$, the above result is no longer true. Indeed, if $\cA_0=[N/2,N] \cup \{1\},$ then for any sequence $\{a_n\}_{n=1}^{N}$ in $\{-1,1\}$,
\[
\left|\sum_{n \in \cA_0}\frac{a_n}{n} \right| \gg 1.
\]
Moreover, we cannot replace $\liminf$ with $\limsup$ in the statement of Theorem \ref{t:postive_density_harmonic_series}, that is, replace {\it positive lower density} with {\it positive upper density}. To see this, we take $\cA_0$ as the union of the lacunary intervals
\[
\cA_0 = \bigcup_{i=0}^{\infty}[N_i,p_i],
\]
where $N_0$ is a sufficiently large integer,
\[
N_{i}=e^{e^{N_{i-1}^2}}, \quad \text{and} \quad p_i=p_{\pi(2N_i)}. 
\] Note that $\cA_0$ has upper density $\geq 1/2.$ Moreover, if $N\in [\frac{1}{3}N_i, \frac{1}{2}N_i]$, then by the prime number theorem,
\[
\left|\sum_{n \in \cA_0 \cap [N]} \frac{a_n}{n} \right| = \left| \sum_{n \in \cA_0 \cap [p_{i-1}]}\frac{a_n}{n} \right| \gg \frac{1}{\exp\left(\sum_{n \leq p_{i-1}}\Lambda(n) \right)}\gg \frac{1}{e^{(1+o(1))p_{i-1}}}>\frac{1}{(\log N)^{100}}.
\]

In view of the above, we show that if $\mathcal{A}_0$ has positive upper density, then the analog of Theorem \ref{t:postive_density_harmonic_series} holds along subsequences.

\begin{theorem}\label{t:postive_upp_density_harmonic_series}
Let $\cA_0 \subset \N$ have positive upper density, that is, for some $0<\delta<1$,
\[
\limsup_{N \to \infty}\frac{|\cA_0 \cap [N]|}{N}>\delta.
\]
There exists an infinite sequence of integers $\{N_i\}_{i=0}^{\infty}$ and a sequence $\{a_n\}_{n=1}^{\infty}$ in $\{-1,1\}$ such that, for all $i \geq 0$, 
\[
\left|\sum_{n \in \cA_0 \cap[N_i]} \frac{a_n}{n} \right|\leq \exp\left(-N_i^{\theta(\delta)} \right), \quad \text{with } \theta(\delta)>0.
\]
\end{theorem}
We now briefly highlight some general ideas that underpin the proofs of our results. By choosing a random sequence $\{a_n\}_{n=1}^{N}$ uniformly distributed in $\{-1,1\}$, our aim is to show that
    \[
    \mathbb{P}\big( |{\bf X}_{\cA}-x_0| \leq \eta \big)>0.
    \]
holds with positive probability for appropriately chosen small (but not too small!) values of $x_0$ and $\eta$. The problem is then reduced to studying the Fourier decay of the corresponding density function, which takes the form
\[
\hat{\rho}_{\cA}(t) = \prod_{n \in \cA} \cos \left( \frac{2\pi t}{n} \right), \quad \cA \subset \N
\]
for $t$ in an intermediate range, say $t \in (N, U(N))$. The larger $U(N)$ is so that $\hat{\rho}_{\cA}(t) \ll 1/t^2$ holds, the better the upper bound for $\displaystyle{\sum_{n \leq N, n \in \cA}a_n/n}$ can be obtained. We begin by proving a key general Lemma \ref{l:key_local_lemma}, which ensures that if for a set $\cA$ one can select a large subset $\cB,$ satisfying certain multiplicative properties, then one can guarantee suitable decay of $\hat{\rho}_{\cA}(t).$ We then use “anatomy of integers” arguments related to smooth and rough numbers (see Lemma \ref{l:multi_str_positive_density_set}) to show that such a choice is possible if $\cA$ is any sufficiently dense subset of integers.

The idea of proving Theorem \ref{t:small_log_mean_mult_f} is based on a decomposition of $L(f,N)$ into two deterministic and one random parts. We start with the Liouville function $\lambda(n),$ which, by the prime number theorem, gives reasonably small partial sums 
\[\sum_{n\le N}\frac{\lambda(n)}{n}=O\left(\frac{1}{\log^A N}\right).\]
We then use a deterministic part of the form  
\[
\sum_{c_1 N<p \leq c_2 N} \frac{\lambda(p)}{p}, \quad \text{with} \quad 0<c_1<c_2\leq 1, 
\]
to flip the signs of $\lambda(p)=-1,$ if necessary, to end up with the function $\tilde{\lambda}(n),$ satisfying
\[\sum_{n\le N}\frac{\tilde{\lambda}(n)}{n}\asymp \frac{1}{N}.\]
We finally use a random part of the form  
\[
\sum_{\alpha N<p \leq \beta N} \frac{f(p)}{p}, \quad \text{with} \quad 0<\alpha<\beta\leq 1, 
\]
where we take $f(p)$ to be uniformly distributed $\pm 1$ random variables and apply Lemma \ref{l:key_local_lemma} to get ``close" to the deterministic value and  finish the proof.

In order to construct a suitable infinite sequence $N_i\to\infty,$ in all our results above, we perform  an inductive procedure in conjunction with analytic arguments based on Perron's formula to show that we can ``glue" well-separated scales together. We remark that all examples of multiplicative $f$ satisfying the conclusion of Theorem \ref{t:small_log_mean_mult_f} must necessarily be ``pretentious" to $\lambda(n)$ in a precise sense of Granville and Soundararajan; see recent work \cite{KM} for further details.\\

The remainder of the paper is organized as follows. In Section \ref{S:key_local}, we prove a key local Lemma \ref{l:key_local_lemma} and illustrate its power by reproving \cite[Corollary 15]{GTZ20}. We then turn to the proof of Theorem \ref{t:small_log_mean_mult_f}. In Section \ref{S:pf_thms}, we state useful bounds for smooth numbers and extract multiplicative structures from sets with large density by proving a key arithmetic Lemma \ref{l:multi_str_positive_density_set}. Finally, we prove Theorem \ref{t:postive_density_harmonic_series} and Theorem \ref{t:postive_upp_density_harmonic_series}.

    \section{Small scale distribution of random harmonic sums}\label{S:key_local}
We begin with the key technical lemma which we shall apply repeatedly in the sequel.
    \setcounter{lemma}{0} \setcounter{theorem}{0}
    \setcounter{equation}{0}
    \begin{lemma}\label{l:key_local_lemma}
       Let $\{a_n\}_{n=1}^{N}$ be independent and uniformly distributed random variables in $\{-1,1\}$. For $0<c<1$ and $\cA \subset [cN, N]$, we define the harmonic sum 
   $$ {\bf X}_{\cA}=\sum_{n \in \cA}\frac{a_n}{n}.
    $$ Assume that there exists a subset $\cB  \subset \cA$ with the following properties:
    
    \begin{enumerate}
        \item $|\cB| \geq N^{2/3+\epsilon}$, for some $0<\epsilon \leq 1/3;$
        \item There exist $k \ge 1$ and a set $R(\cB) \subset \N$ such that $$\Omega(R(\cB)):=\max_{n \in \cB}\Omega(n) \leq k$$ and any $b \in \cB$ can be written uniquely $b=rs$ with $r \in R(\cB)$. 
    \end{enumerate}
  Let  $C=\frac{3\log (|\cB|)- 2\log N}{\log N}>\epsilon$ and assume that 
     \begin{equation}\label{e:eta}
   \frac{1}{N} \geq \eta \geq \exp\left( -\left(\frac{C^{2k}}{100}\frac{|\cB|^3}{N^2} \right)^{\frac{1}{2k+1} } (\log N)^{\frac{2k}{2k+1}}\right).
    \end{equation} 
       Then for $x_0 \ll \frac{1}{N}$, we have  
    \[
    \mathbb{P}\big( |{\bf X}_{\cA}-x_0| \leq \eta \big) \gg \frac{N}{\sqrt{|\cA|}} \eta.
    \] In particular, there exists a sequence $\{a_n\}_{n=1}^{N}$ in $\{-1,1\}$ such that 
    \[  \left\vert \sum_{n \in \cA} \frac{a_n}{n} -x_0 \right\vert \leq \eta. \]
    \end{lemma}

    \begin{proof}
    Let $\phi(x)$ be a smooth even function supported on $[-1, 1]$ such that $\phi(x)=1$ when $x \in [-1/2, 1/2]$, $0 \leq \phi(x) \leq 1$ and satisfying $\phi^{(j)}(x) \ll_j 1$ for all $j \geq 1$. 
    We set
    \[
    \psi(x): = \phi(x/\eta),
    \] which is an even function and verifies $ \psi^{(j)}(x) \ll_j \frac{1}{\eta^j}.$
   We define the Fourier transform of $\psi$ by
    \[
    \widehat{\psi}(t):=\int_{\R} \psi(x) e^{-2i\pi xt} dx,
    \]
    which satisfies the following properties:
    \begin{enumerate}
        \item 
        \[
        \widehat{\psi}(t) = \int_{-\eta}^{\eta}\psi(x) e^{-2i \pi t x}dx = 2 \int_{0}^{\eta}\psi(x) \cos{(tx)} dx \in \R,
        \]
        and 
        \begin{equation}\label{e:low_hat_psi}
         \widehat{\psi}(t)> \frac{9}{10} \eta \text{ when } |t| \leq \frac{1}{100 \eta};
        \end{equation}
        \item Applying integration by parts to $\widehat{\psi}(t)$, one can obtain that for all $j \geq 0$ and $t \in \R$,
        \[
        \widehat{\psi}(t) \ll_j \frac{1}{\eta^{j-1}|t|^j} .
        \]
        In particular,
        \begin{equation}\label{e:upbd_hat_psi}
        \widehat{\psi}(t) \ll \min \left\{ \eta, \frac{1}{\eta |t|^2} \right\}. 
        \end{equation}
        
    \end{enumerate}
Let us define the even function $$\rho_{\cA}(t):=\mathbb{E}[e^{2i \pi t{\bf X}_{\cA}}]= \prod_{n \in \cA} \cos{\left(\frac{2\pi t}{n} \right)}.$$ Using Fourier inversion and the parity of $\widehat{\psi}$ we have 
\begin{align}\label{e:prob_int}
\mathbb{P}\big( |{\bf X}_{\cA}-x_0|\leq \eta \big)& = \mathbb{E}(\mathbf{1}_{\vert {\bf X}_{\cA}-x_0 \vert \leq \eta}) \geq  \mathbb{E}(\psi({\bf X}_{\cA}-x_0)) \nonumber \\
& =\frac{1}{2^{\vert\cA \vert }} \sum_{\delta_1, \dots, \delta_{\vert \cA\vert} \in \left\{\pm 1\right\}} \psi\left(\sum_{n \in \cA} \frac{\delta_n}{n} - x_0  \right)  \nonumber  \\ 
& =  \int_{\mathbb{R}} \widehat{\psi(t)} \rho_{\cA}(t) e^{-2i\pi tx_0} dt  = 2\int_{0}^{+\infty} \cos{ (2 \pi t x_0)}\widehat{\psi}(t)  \rho_{\cA}(t)  dt.
\end{align}

    
  It remains to lower bound the last integral in \eqref{e:prob_int}. We split the domain of integration into four parts.  If $0 \leq |t| \leq \frac{N}{\sqrt{|\cA|}}$, we have 
    \begin{align*}
    &\rho_{\cA}(t)= \prod_{n \in \cA} \cos{\left(\frac{2 \pi t}{n} \right)} = \prod_{n \in \cA} \left( 1 - \frac{2\pi t^2}{n^2} + O\left(\frac{t^4}{n^4} \right)\right)\\
    & = \exp\left({-\sum_{n \in \cA} \left(\frac{2 \pi t^2}{n^2} + O\left(\frac{t^4}{n^4} \right) \right)} \right).
    \end{align*}
    Note that
    \[
    \sum_{n \in \cA}\frac{t^2}{n^2} \leq \frac{N^2}{|\cA|}\frac{1}{c^2N^2}|\cA| = \frac{1}{c^2} \quad \text{and} \quad \sum_{n \in \cA}\frac{t^4}{n^4} \ll \frac{1}{|\cA|},
    \]
    so $$\rho_{\cA}(t) \geq e^{-\frac{2\pi}{c^2}+o(1)} \gg 1.$$ By (\ref{e:low_hat_psi}) and using $\frac{N}{\sqrt{|\cA|}} \leq \frac{1}{100 x_0}$, we get
     \[
    \int_{0 \leq |t| \leq \frac{N}{\sqrt{|\cA|}}} \cos{( 2\pi t x_0)} \widehat{\psi}(t) \rho_{\cA}(t) dt \gg \frac{N}{\sqrt{|\cA|}} \eta.
    \]
   If $$\frac{N}{\sqrt{|\cA|}} < t < T_0:=\min \{ \frac{1}{4x_0}, \frac{1}{4}cN\},$$ we notice that $\cos{(2 \pi t x_0)}>0$ and $\rho_{\cA}(t)>0$. Moreover using (\ref{e:low_hat_psi}), we obtain
    \[
    \int_{\frac{N}{\sqrt{\cA}} < t < T_0}  \cos{(2\pi tx_0)}\widehat{\psi}(t) \rho_{\cA}(t) dt >0.
    \]
    Now, we study the case when 
    \[
    T_0 \leq  t\leq T:= \exp\left( \left(\frac{C^{2k}}{32}\frac{|\cB|^3}{N^2} \right)^{\frac{1}{2k+1} } (\log N)^{\frac{2k}{2k+1}}\right).
    \]
    Note that \begin{equation}\label{cos_bdd}
    \vert \cos{(2 \pi \theta)} \vert \leq \exp{\left(-2\pi^2\|\theta\|^2 \right)} \quad \text{and} \quad |\rho_{\cA}(t)|= \left|\prod_{n \in \cA} \cos{\left(\frac{2\pi t}{n} \right)}\right|\leq \prod_{n \in \cB} \left|\cos{\left(\frac{2 \pi t}{n} \right)}\right|
    \end{equation} where $\|x\|$ denotes the distance of $x \in \mathbb{R}$ from its nearest integer. Thus we have 
    \begin{align*}
    \vert \rho_{\cA}(t) \vert & \leq \exp{ \left( -2 \pi^2 \sum_{n \in \cB} \left\|\frac{t}{n} \right\|^2\right)}   \leq \exp{ \left( -2(\pi\delta)^2 \sum_{n \in \cB} {\bf 1}_{\left\|\frac{t}{n} \right\|> \delta}\right)}  \\
    & \leq \exp{ \left( -2(\pi\delta)^2 \left(|\cB|- |S_{\cB,\delta}(t)| \right)\right)},
    \end{align*}
    where $\delta$ is a parameter to be chosen and
    \[
   S_{\cB,\delta}(t)= \{n \in \cB: \| t/n \| \leq \delta \}.
    \] We argue similarly as in \cite[Lemma $3.3$]{BMS18} to bound $S_{\cB,\delta}(t).$
    \begin{align*}
    |S_{\cB,\delta}(t)|  &= \#\{n \in \cB: \exists l \in \Z:l-\delta < t/n < l+\delta\}
    \\
    &=\#\{n \in \cB: \exists l \in \Z:t-\delta n< ln < t + \delta n\}\\
     & \leq \#\{n \in \cB: \exists l \in \Z:t-\delta N< ln < t + \delta N\}\\
     & = \sum_{n \in \cB} \sum_{\substack{t -\delta N< m < t + \delta N \\ n \mid m}} 1= \sum_{t -\delta N< m < t + \delta N }\sum_{\substack{n \mid m \\ n \in \cB}}  1.
    \end{align*}
    By the definition of $R(\cB)$, we can bound the above by 
    \[
    \sum_{t -\delta N< m < t + \delta N }\sum_{\substack{n \mid m \\ n \in R(\cB)}}  1 \leq \sum_{j=1}^{k} \sum_{t -\delta N< m < t + \delta N } \left(\sum_{p \mid m}1 \right)^j \leq k\sum_{t -\delta N< m < t + \delta N } \omega(m)^k,
    \]
    where $\omega(m)$ denotes the number of distinct number of prime factors of $m$. By the known bound $\omega(m) \leq \frac{\log m}{\log \log m}$, we finally obtain that 
    \[
    |  S_{\cB,\delta}(t)| \leq 2k\delta N \left( \frac{\log t}{\log \log t} \right)^k,
    \]
    since $t \gg N$.

    By choosing $$\delta = \frac{1}{4k} \left(\frac{\log \log t}{\log t} \right)^k \frac{|\cB|}{N},$$  we have
    \[
    |  S_{\cB,\delta}(t)| \leq \frac{1}{2} |\cB|, 
    \]
    which implies that, after an easy numerical calculation, when $T_0 < t \leq T$,
    \[
   |\rho_{\cA}(t)| \leq \exp{ \left( -\pi^2 \delta^2|\cB|\right)} \leq \frac{1}{t^2}.
    \]
    Hence,
    \[
    \int_{T_0<t \leq T}|\widehat{\psi}(t)\rho_{\cA}(t)|dt \ll \eta \int_{T_0 <t \leq T} \frac{1}{t^2} dt = o(\eta).
    \]
    If $ t >T$, by (\ref{e:eta}) and (\ref{e:upbd_hat_psi}), then
    \[
    \int_{t > T }|\widehat{\psi}(t)| dt \leq \frac{1}{\eta}\int_{t > T }\frac{1}{t^2} dt = o(\eta).
    \]
    \end{proof}
    
    \begin{Rem}
    Choosing $x_0$ to be small enough is the key to ensuring  positivity in the range $\frac{N}{\sqrt{\cA}} \leq t \ll N$. 
    \end{Rem}

\section{Application to signed harmonic sums and proof of Theorem \ref{t:small_log_mean_mult_f}}
Before moving on to the proof of Theorem \ref{t:small_log_mean_mult_f}, we illustrate how to apply Lemma \ref{l:key_local_lemma} to obtain strong results for sets $\cA \subset (N,2N)$ of zero asymptotic density.

\begin{lemma}\label{typical-nb-factors}
Let $\cA \subset (N,2N)$ such that $|\cA| \gg \frac{N}{(\log \log N)^C}$ for some constant $C>0.$ There exists a subset $\mathcal{B}$ of $\cA$ such that $|\mathcal{B}| \geq \frac{|\cA|}{2}$ and 
$$ \Omega(\mathcal{B}):= \max_{n \in \mathcal{B}} \Omega(n) \leq 2 \log \log N.$$

\end{lemma}

\begin{proof}
By \cite[Theorem $1$]{GSE07}, we have
$$ \sum_{n \leq N \atop \Omega(n) \geq 2 \log \log N}1 \leq \sum_{n \leq N} \left(\frac{\Omega(n)-\log \log N}{\log \log N}\right)^{2k} \ll_k \frac{N}{(\log \log N)^k}.$$ The result follows taking $k>C.$
\end{proof}

 \begin{corollary}\label{cor:loglog}
     Let $\cA \subset (N,2N)$  such that $$|\cA| \gg \frac{N}{(\log \log N)^C}$$ for some constant $C>0.$  There exist a constant $c>0$ and a sequence $\{a_n\}_{n=1}^{N}$ in $\{-1,1\}$ such that 
    \begin{equation}\label{weakerbnd}  \left\vert \sum_{n \in \cA} \frac{a_n}{n}  \right\vert \leq \exp\left(-\exp\left(c\frac{\log N}{\log \log N}\right)\right). \end{equation}
 \end{corollary}
\begin{proof}
    Taking $\cB$ as in Lemma \ref{typical-nb-factors}, the result follows from an application of Lemma \ref{l:key_local_lemma} and a simple computation.
\end{proof}
\begin{Rem}
We could allow $\cA$ to have a smaller density choosing a subset $\cB \subset \cA$ satisfying $\Omega(\cB) \leq (\log N)^{\epsilon}.$ For instance, we could replace the condition in Corollary \ref{cor:loglog}
by $|\cA| \gg \frac{N}{(\log  N)^C}$ replacing the right hand side of \eqref{weakerbnd} by the weaker bound
$\exp\left(-\exp\left((\log N)^{1-\varepsilon}\right)\right)$.

\end{Rem}

We now introduce a simple ``flipping trick'', which will be applied throughout the paper. 
    \begin{lemma}[Flipping trick]\label{l:flip}
    Given a set $S \in [c_1N,c_2N]$ with $0<c_1<c_2$ and $\alpha \in \R$, such that 
    \[
    \sum_{n \in S} \frac{1}{n} > |\alpha|,
    \]
    there exists a sequence $\{b_n\}_{n \in S}$ in $\{-1,1\}$ such that 
    \[
   \sum_{n \in S}\frac{b_n}{n} -\alpha  \leq \frac{1}{c_1 N}.
    \]
    \end{lemma}

\begin{proof}
Let $S=\{n_1,\dots,n_k\}$ and assume without loss of generality that $\alpha>0$. Let $$j_0:=\min\{1\leq j \leq k, \sum_{i=1}^{j}\frac{1}{n_i}>\alpha\}.$$  Clearly $$0< \sum_{i=1}^{j_0} \frac{1}{n_i} -\alpha \leq \frac{1}{n_{j_0}} \leq \frac{1}{c_1 N}.$$ We then use the greedy algorithm to ensure that the error decreases at each step.
\end{proof}


     We will show that ``global information'' can be obtained from ``local information'' by proving the following corollary. 
    \begin{corollary}(\cite[Corollary 15]{GTZ20})\label{c:small_harmonic_series}
    Let $N$ be a sufficiently large integer. There exists a sequence $(a_n)_{1 \leq n \leq N} \in \{-1,1\}$ such that 
    $$ 
    \sum_{n=1}^{N}\frac{a_n}{n} \ll \exp{\left(- \frac{1}{1000}\frac{N^{1/3}}{(\log N)^{1/3}} \right)}.
    $$
    \end{corollary}

    \begin{proof}
    We first claim that there exists a sequence $a_n \in \{-1,1\}$ such that 
    \begin{equation}\label{e:constr_a_n_n_for_x_0}
    \sum_{1 \leq n \leq N/2}\frac{a_n}{n} \ll \frac{1}{N}.
    \end{equation}
    Take $a_{2j-1}=1$ and $a_{2j}=-1$ for $j \leq J=\lfloor N/4e^2 \rfloor$ and let 
    $$\alpha:=  \sum_{n=1}^{2J} \frac{a_n}{n}.$$ We have  
    \[
    \frac{1}{2}<\alpha=1-\frac{1}{2}+\frac{1}{3}-\frac{1}{4}+\cdots-\frac{1}{2J} \leq \sum_{n=1}^{\infty}\left(\frac{1}{n}-\frac{1}{n+1}\right)=1.
    \]
    Note that 
    \[
    \sum_{2J<n \leq N/2}\frac{1}{n} = \log (2e^2) - \log 2 + O\left(\frac{1}{N} \right) = 2 + O\left(\frac{1}{N} \right).
    \]
    By Lemma \ref{l:flip}, there exists a sequence $a_n\in \{+1,-1\}$ for $n \in (2J, N/2]$ such that
    \[
    \sum_{1 \leq n \leq N/2} \frac{a_n}{n} = \alpha + \sum_{2J<n \leq N/2}\frac{a_n}{n}  \ll \frac{1}{N}.
    \] 
We let
    $$ x_0 =-\sum_{1\leq n \leq N/2}\frac{a_n}{n},$$ by the prime number theorem, we can apply Lemma \ref{l:key_local_lemma} with the following choice of parameters:
    \[
 \cA=[N/2,N], \cB=R(\cB)=\{p\in \cA: p \text{ prime}\}, \text{ and } k=1.
    \]
  We obtain that there is a sequence $\{a_n\}_{N/2}^{N}$ in $\{-1,1\}$, such that 
    \[
    \left|\sum_{1 \leq n\leq N}\frac{a_n}{n} \right| = \left|\sum_{N/2<n \leq N}\frac{a_n}{n} + \sum_{1 \leq n \leq N/2} \frac{a_n}{n}\right| \leq \eta,
    \]
    where 
    \[
    \eta \geq \exp\left( - \frac{1}{1000} \left(\frac{N}{\log N} \right)^{1/3}\right).
    \]
    \end{proof}

    \begin{Rem}
    The idea of decomposing $[N]$ into two intervals to construct one deterministic part and one random part will also be applied to the proof of Theorem \ref{t:small_log_mean_mult_f}. Instead of decomposing the deterministic part by hand, one could have applied \cite[Theorem 1.1]{BMS18}, which shows that for $1<n \leq N/2$, there exists a sequence $a_n \in \{-1,1\}$ such that
    \[
    \sum_{1 \leq n \leq N/2} \frac{a_n}{n} \ll \exp\left( - (\log N)^2 \right).
    \]
    Taking $\alpha = -\sum_{1 \leq n \leq N/2} \frac{a_n}{n}$ in Lemma \ref{l:key_local_lemma}, Corollary \ref{c:small_harmonic_series} follows by the same arguments as above. 
    \end{Rem}

    We are now ready to prove Theorem \ref{t:small_log_mean_mult_f}.
    \begin{proof}[Proof of Theorem \ref{t:small_log_mean_mult_f}]
    We will construct by induction an increasing sequence $N_0<N_1<\dots$ and a multiplicative function $f$ taking values in  $\{-1,1\}$ such that  \[
    \left| \sum_{1 \leq n \leq N_i}\frac{f(n)}{n} \right| \ll \exp\left(-c_0 \frac{N_i^{1/3}}{(\log N_i)^{1/3}} \right).
    \] To do so, we judiciously modify the Liouville function on some well-chosen intervals (this idea already appears in \cite{GS07} in connection with negative truncation problem of $L(1,\chi_d)$). Let us start with the initial step. By Haselgrove's result \cite{Ha58}, we can consider the first integer $C>1$ such that
    \[
     -\Delta:=\sum_{1 \leq n \leq C} \frac{\lambda(n)}{n} <0.
    \]
    By a simple observation, $0 < \Delta <\frac{1}{C}$.
    For a sufficiently large $N_0$, we write
    \begin{align*}
    \sum_{1 \leq n \leq N_0} \frac{\lambda(n)}{n} & = \sum_{\frac{N_0}{2}<p \leq N_0} \frac{\lambda(p)}{p} + \sum_{\frac{N_0}{C+1}<p \leq \frac{N_0}{C}}\frac{\lambda(p)}{p} \sum_{1 \leq n \leq C} \frac{\lambda(n)}{n} + E\\
    & = \sum_{\frac{N_0}{2}<p \leq N_0} \frac{\lambda(p)}{p} - \Delta \sum_{\frac{N_0}{C+1}<p \leq \frac{N_0}{C}}\frac{\lambda(p)}{p}+E\\
    & = -\sum_{\frac{N_0}{2}<p \leq N_0} \frac{1}{p} + \Delta \sum_{\frac{N_0}{C+1}<p \leq \frac{N_0}{C}}\frac{1}{p} + E
    \end{align*}
    By the prime number theorem, we have 
    \[
    \sum_{ N_0/2 <p \leq N_0} \frac{1}{p} = \frac{\log2}{\log N_0} + O\left( \frac{1}{\log^2 N_0}\right) \quad \text{and} \quad \sum_{\frac{N_0}{C+1}<p \leq \frac{N_0}{C}}\frac{1}{p}=\frac{\log(1+1/C)}{\log N_0}+ O\left( \frac{1}{\log^2 N_0}\right).
    \]
    Since 
    \[
    \log 2 > \frac{1}{C} \log \left(1+\frac{1}{C} \right )>\Delta \log \left(1+\frac{1}{C} \right),
    \]
    we have, by the upper bound \eqref{e:Liouvill_pnt},
    \[
    E  = (\log 2 - \Delta \log (1+1/C)) \frac{1}{\log N_0} +  O\left( \frac{1}{(\log N_0)^2}\right).
    \]
    By Lemma \ref{l:flip}, we can find $f(p) \in \{-1,1\}$ for $p \in (N_0/2,N_0]$ such that  
    \[
    E+\sum_{N_0/2<p \leq N_0}\frac{f(p)}{p} \ll \frac{1}{N_0}.
    \]
   We now apply Lemma \ref{l:key_local_lemma} with 
      \[
    \cA=\cB=R(\cB)=\{p \in [N_0/(C+1),N_0/C]: p \text{ prime}\}
    \] and 
    \[
    -x_0 = \frac{1}{\Delta}\left(E+\sum_{N_0/2<p \leq N_0}\frac{f(p)}{p} \right).
    \] Thus, there is a sequence $\{f(p)\}_{p=N_0/(C+1)}^{N_0/C}$ in $\{-1,1\}$ and $c_0>0$ small, such that 
    \[
    \left|\sum_{p=N_0/(C+1)}^{N_0/C}\frac{f(p)}{p}-x_0 \right| \ll \exp\left( - c_0\left(\frac{N_0}{\log N_0} \right)^{1/3}\right).
    \] 
    Assume now that $f(p)$ has been defined for $p \leq N_k$ modifying the Liouville function $\lambda$ for primes $$p \in J_i:=[N_i/2, N_i]\cup[N_i/(C+1),N_i/C]$$ for $0 \leq i \leq k$ in order to have
    \[
    \left| \sum_{1 \leq n \leq N_i}\frac{f(n)}{n} \right| \ll \exp\left(-c_0 \frac{N_i^{1/3}}{(\log N_i)^{1/3}} \right), \qquad 1 \leq i  \leq k.
    \]
    Let $\lambda^*$ be a completely multiplicative function defined as
    \[
    \lambda^*(p)= 
    \begin{cases}
    f(p) & \text{if }1 \leq p \leq N_k;\\
    -1 & \text{otherwise}.
    \end{cases}
    \]
    Let $N_{k+1} = 2^{N_{k}}$. Applying Perron's formula with $b>0$, we get
    \begin{equation}\label{Perron}
    \sum_{1 \leq n \leq N_{k+1}}\frac{\lambda^*(n)}{n} = \frac{1}{2 \pi i} \int_{b-i N_{k+1}}^{b+ i N_{k+1}}  G(s+1)\frac{N_{k+1}^s}{s} ds + O\left(\frac{\log N_{k+1}}{N_{k+1}}\right)
    \end{equation}
    where \[G(s) = \frac{\zeta(2s)}{\zeta(s)}\prod_{p\in \bigcup_{i=0}^{k}J_i}\left(1+\frac{1}{p^s} \right)\left(1+\frac{f(p)}{p^s} \right)^{-1}.
    \] The Vinogradov-Korobov zero free-region for $\zeta(s)$ allows us to shift the integral in \eqref{Perron}, without encountering any pole of $G(s)$, to $\Re(s)=b'$ with
    \[
    b'=-\frac{B}{ (\log N_{k+1})^{\frac{2}{3}}(\log \log N_{k+1})^{\frac{1}{3}}},
    \] where $B>0$ is an absolute constant. 
   For $N_0$ sufficiently large, we have
    \[
    \prod_{p\in \bigcup_{i=0}^{k}J_i}\left(1+\frac{1}{p^s} \right)\left(1+\frac{f(p)}{p^s} \right)^{-1} \ll \prod_{p\in \bigcup_{i=0}^{k}J_i}\left(1+\frac{1}{p} \right)^2 \leq \prod_{p\in J_0}\left(1+\frac{1}{p} \right)^{2k} \leq e^{2k}.
    \] Notice that by construction $k \leq \log \log  N_{k+1}.$ Thus, arguing similarly as in the proof of the prime number theorem, the integral over the horizontal lines are easily bounded by $\ll (\log N_{k+1})^{-A}$ for any $A \geq 1.$ We obtain that for any $A \geq 1$, 
    \begin{align*}
    \sum_{1 \leq n \leq N_{k+1}}\frac{\lambda^*(n)}{n} & \ll  N_{k+1}^{b'}\int_{b'-i N_{k+1}}^{b'+ i N_{k+1}}  |G(s+1)|\frac{1}{|s|} ds + O\left(    \frac{1}{(\log N_{k+1})^A}+\frac{\log N_{k+1}}{N_{k+1}}\right) \\
    & \ll \exp(- (\log N_{k+1})^{1/3-\varepsilon}) + O\left(    \frac{1}{(\log N_{k+1})^A}\right) \ll \frac{1}{(\log N_{k+1})^A}.
    \end{align*}
    
    We select $f(p)$ for $N_k < p \leq N_{k+1}$ following the same lines as in the base case. We write 
    \begin{align*}
    \sum_{1 \leq n \leq N_{k+1}} \frac{\lambda^*(n)}{n} & = \sum_{\frac{N_{k+1}}{2}<p \leq N_{k+1}} \frac{\lambda(p)}{p} + \sum_{\frac{N_{k+1}}{C+1}<p \leq \frac{N_{k+1}}{C}}\frac{\lambda(p)}{p} \sum_{1 \leq n \leq C} \frac{\lambda(n)}{n} + E^*\\
    & = \sum_{\frac{N_{k+1}}{2}<p \leq N_{k+1}} \frac{\lambda(p)}{p} - \Delta \sum_{\frac{N_{k+1}}{C+1}<p \leq \frac{N_{k+1}}{C}}\frac{\lambda(p)}{p}+E^*\\
    & = -\sum_{\frac{N_{k+1}}{2}<p \leq N_{k+1}} \frac{1}{p} + \Delta \sum_{\frac{N_{k+1}}{C+1}<p \leq \frac{N_{k+1}}{C}}\frac{1}{p} + E^*.
    \end{align*}
    Similarly,
    \[
    E^*  = (\log 2 - \Delta \log (1+1/C)) \frac{1}{\log N_{k+1}} +  O\left( \frac{1}{(\log N_{k+1})^2}\right).
    \] Again, by Lemma \ref{l:flip}, there is $f(p) \in \{-1,1\}$ with $N_{k+1}/2<p\leq N_{k+1}$, such that 
    \[
    E^* + \sum_{N_{k+1}/2<p\leq N_{k+1}}\frac{f(p)}{p} \ll \frac{1}{N_{k+1}}
    \]  As before, we apply Lemma \ref{l:key_local_lemma} with 
      \[
    \cA=\cB=R(\cB)=\{p \in [N_{k+1}/(C+1),N_{k+1}/C]: p \text{ prime}\}
    \] and 
    \[-x_0 = \frac{1}{\Delta}\left(E^*+\sum_{N_{k+1}/2<p \leq N_{k+1}}\frac{f(p)}{p} \right).
    \]
    Therefore, there exists a sequence $\{f(p)\}_{p=N_{k+1}/(C+1)}^{N_{k+1}/C}$ in $\{-1,1\}$ and $c_0>0$ small, such that 
    \[
    \left|\sum_{p=N_{k+1}/(C+1)}^{N_{k+1}/C}\frac{f(p)}{p}-x_0 \right| \ll \exp\left( - c_0\left(\frac{N_{k+1}}{\log N_{k+1}} \right)^{1/3}\right).
    \] This concludes the induction and the proof.
    \end{proof}

    \section{Proofs of theorems \ref{t:postive_density_harmonic_series} and \ref{t:postive_upp_density_harmonic_series}}\label{S:pf_thms}
    \setcounter{lemma}{0} \setcounter{theorem}{0}
    \setcounter{equation}{0}

    
We will need the following estimate for the counting function of smooth numbers due to Dickman (see \cite[Equations (1.1) and (1.12)]{G08}).
    \begin{lemma}
    Let $S(x,y)$ be the set of integers up to $x$, all of whose prime factors are $ \leq y$ (such integers are called $y$-smooth), and let $\Psi(x,y)$ be the number of such integers. For any fixed $u \geq 1$
    \[
    \Psi(x,y) \sim x \rho(u), \quad \text{as } x \to \infty, \quad \text{where } x=y^u
    \]
    where, $\rho$ is the Dickman function and satisfies the estimate 
    \[
    \rho(u) = \frac{1}{u^{u(1+o(1))}}, \quad \text{ as } u \to \infty.
    \]
    \end{lemma}

    \begin{Rem}
    Before proving the above lemma, we remark that in the general case, one cannot expect the exponent $\theta(\delta)$ to be replaced by $1/3$. Indeed, taking 
    \[
    \cA = \Psi(N,N^{1/10}) \cup \{p_1\}, \quad \text{with} \quad N/2<p_1 \leq N,
    \]
    we have, by the prime number theorem,
    \[
    \left| \sum_{n \in \cA} \frac{a_n}{n} \right| \geq \frac{1}{\text{lcm}\{a:\ a\in A\}} \geq \frac{1}{p_1\prod_{p \leq N^{1/10}}p} \geq \frac{1}{\exp(2N^{1/10})}.
    \]
 This means that, up to the power of $N$, our result is optimal.
    \end{Rem}
    To apply Lemma \ref{l:key_local_lemma} successfully, we need to find a subset  $\cB \subset\cA$ and the corresponding set $R(\cB)$ with $\Omega(R(\cB)) \ll 1$. In the following lemma, we construct  $R(\cB)$ as the ``rough part'' of $\cB$ (or $\cA$). For any $n\in \cA$ and $\epsilon>0$, we can write uniquely $n=n_{r,\epsilon}n_{s,\epsilon}$ where $P^{-}(n_{r,\epsilon})>N^{\epsilon} $ is the $\epsilon$-rough part of $n$ and $P^{+}(n_{s,\epsilon}) \leq N^{\epsilon}$ is the $\epsilon$-smooth part of $n$.
    \begin{lemma}\label{l:multi_str_positive_density_set}
    Let $N$ be a sufficiently large integer, and $0<\delta \leq 1$. Suppose that $\cA \subset [N]$ with $|\cA| \geq \delta N$.
    For any $\epsilon_0>0$, there exists $0<\epsilon_1<1$ such that 
   \[
    |R_{\epsilon_1}(\cA)| \geq   N^{1-\epsilon_0},
    \]
    where $
    R_{\epsilon_1}(\cA):=\{n_{r,\epsilon_1}, n \in \cA\}$ is the set of $\epsilon_1$-rough parts of $\cA$.
    
    \end{lemma}

    \begin{proof}
    We define the set of $\epsilon_1$-smooth parts of $\cA$ as 
    \[S_{\epsilon_1}(\cA)= \{n_{s,\epsilon_1}, n \in \cA\}.  \]
  We  prove the lemma by contradiction. Assume that 
    \begin{equation}\label{upper_bndR}
    |R_{\epsilon_1}(\cA)| \leq N^{1- \epsilon_0} \quad \text{for all } 0<\epsilon_1<1.
    \end{equation} We clearly have 
    \begin{align*}
        \vert \cA \vert  
        & \leq \sum_{r \in R_{\epsilon_1}(\cA)} \sum_{s \in S_{\epsilon_1}(\cA) \atop rs \in \cA}1  \\
        & \leq \sum_{\substack{r \in R_{\epsilon_1}(\cA) \\ 1 \leq r \leq N^{1-\epsilon_0/2}}} \sum_{s \in S_{\epsilon_1}(\cA)}1 + \sum_{\substack{r \in R_{\epsilon_1}(\cA) \\ N^{1-\epsilon_0/2} < r \leq N }} \sum_{\substack{s \in S_{\epsilon_1}(\cA) \\ 1 \leq s \leq N^{\epsilon_0/2}}}1.
    \end{align*}
    By \eqref{upper_bndR}, the second term of the above is bounded by
    \[
    |R_{\epsilon_1}(\cA)|N^{\epsilon_0/2} \leq N^{1-\epsilon_0/2} =     o(|\cA|). 
    \]
    Hence,
    \begin{equation}\label{e:low_bdd_small_rough_part}
    \sum_{\substack{r \in R_{\epsilon_1}(\cA) \\ 1 \leq r \leq N^{1-\epsilon_0/2}}} \sum_{s \in S_{\epsilon_1}(\cA)}1 \geq (1+o(1))|\cA| \geq (\delta+o(1))N.   
    \end{equation}
    On the other hand,
    \begin{align*}
    \sum_{\substack{r \in R_{\epsilon_1}(\cA) \\ 1 \leq r \leq N^{1-\epsilon_0/2}}} \sum_{s \in S_{\epsilon_1}(\cA)}1 & \leq \sum_{k=1}^{\infty}\sum_{\substack{N^{\epsilon_1} \leq p_1 \leq \dots \leq p_k \leq N^{1-\epsilon_0/2} \\ N^{\epsilon_1} \leq p_1\cdots p_k \leq N^{1-\epsilon_0/2}}} \Psi\left(\frac{N}{p_1\cdots p_k},N^{\epsilon_1} \right)\\
    & \leq N\rho\left(\frac{\epsilon_0}{2 \epsilon_1} \right) \sum_{k=1}^{\infty}\left(\sum_{\substack{N^{\epsilon_1} \leq p_1,\dots,p_k \leq N^{1-\epsilon_0/2}}} \frac{1}{p_1 \cdots p_k} \right)\\
    & \leq N \rho\left(\frac{\epsilon_0}{2 \epsilon_1} \right) \prod_{N^{\epsilon_1}\leq p \leq N^{1-\epsilon_0/2}}\left( 1- \frac{1}{p}\right)^{-1}\\
    & \leq N\left(\frac{\epsilon_0}{2 \epsilon_1} \right)^{-\left(\frac{\epsilon_0}{2 \epsilon_1} \right)(1+o(1))} \frac{1-\epsilon_0/2}{\epsilon_1} = o_{\epsilon_1 \to 0}(N)
    \end{align*}
which contradicts the lower bound in (\ref{e:low_bdd_small_rough_part}).
    \end{proof}

\begin{lemma}\label{l:deter_low_upp_large_bdd}
\begin{enumerate}
    \item Let $\cA \subset \N$ have lower density greater than $\delta>\delta_0>0$. There exists a sequence $\{a_n\}_{n=1}^{\infty}$ in $\{-1,1\}$ and a sufficiently large integer $M$ such that
    \[
    \left|\sum_{n \in \cA \cap[N]} \frac{a_n}{n} \right|\leq \frac{2}{\delta_0}\frac{1}{N}, \qquad N>M.
    \]
    \item Let $\cA \subset \N$ have upper density $\delta> \delta_0>0$. There exists a sequence of integers $(P_i)_{i=0}^{+\infty}$ such that $|\cA \cap [P_i]| \geq \delta_0 P_i$ and \[
    \left|\sum_{n \in \cA \cap[P_i]} \frac{a_n}{n} \right|\leq \frac{2}{\delta_0}\frac{1}{P_i}.
    \]
    
    
\end{enumerate}

\end{lemma}
\begin{proof}
We first prove $(1)$. By the definition of limit inferior, there exists a large integer $N_0$ such that $|\cA \cap [N]| \geq \delta_0 N$ for all $N \geq N_0$. First note that by applying the greedy algorithm, one can easily show that there exists a sequence $\{a_n\}_{n=1}^{N_0}$ in $\{-1,1\}$ such that 
\begin{equation}\label{e:bdd_harmonic_sum_positive_lower_density}
\left |\sum_{n \in \cA \cap [N_0]}\frac{a_n}{n} \right| \leq 1.
\end{equation}
Indeed, let
\[
\cA \cap [N_0] = \{n_1,n_2,\dots n_k\} \quad \text{with} \quad n_i<n_{i+1} \text{ for } 1 \leq i \leq k-1.
\]
We prove (\ref{e:bdd_harmonic_sum_positive_lower_density}) by induction. Let $\vert S_1\vert =\frac{1}{n_1} \leq 1.$  Assume that we already found $\{a_{n_i}\}_{i=1}^{j}$ such that $|S_j| \leq 1$ where $S_j = \sum_{i=1}^{j} \frac{a_{n_i}}{n_i}$.  There exists $a_{n_{j+1}} \in \{-1,1\}$ such that 
\[
|S_{j+1}| = \left|S_j+ \frac{a_{n_{j+1}}}{n_{j+1}}\right| \leq 1.
\] 
Define by induction the sequence $N_i = \lceil \frac{2}{\delta_0} N_{i-1} \rceil$ for $i \geq 1$. Noting that $N_{i-1} \leq \frac{\delta_0}{2} N_i$, we get that, for any $i \geq 1$, 
\begin{equation}\label{low:Ni}
\sum_{N_{i-1}<n\leq N_i \atop n \in \cA}\frac{1}{n}  \geq \sum_{(1- \delta_0/2)N_i < n \leq N_i} \frac{1}{n} \geq \frac{1}{2}\log \left (\frac{1}{1-\delta_0/2} \right).
\end{equation}
Let $I:= 4\left\lceil \log^{-1} \left (\frac{1}{1-\delta_0/2} \right) \right\rceil$, using \eqref{low:Ni}, we get 
\[
\sum_{N_0<n\leq N_I\atop n \in \cA}\frac{1}{n} = \sum_{i=1}^{I}\sum_{N_{i-1}<n\leq N_i \atop n \in \cA}\frac{1}{n} \geq 2.
\] By Lemma \ref{l:flip}, there exists a sequence $\{a_n\}_{n=1}^{N_I}$ in $\{-1,1\}$ such that
 $$\left |\sum_{n \in \cA \cap [N_I]}\frac{a_n}{n} \right| \leq \frac{1}{N_0} \leq \left(\frac{2}{\delta_0}\right)^{I} \frac{1}{N_I}.   $$ 
Choosing $N_0>\frac{2}{-\log(1-\delta_0/2)},$ 
 an easy induction using \eqref{low:Ni} and Lemma \ref{l:flip} implies that for all $j>I$, there exists a sequence $\{a_n\}_{n=1}^{N_{j+1}}$ in $\{-1,1\}$ such that
 \begin{equation}\label{boundN_i} \left |\sum_{n \in \cA \cap [N_{j+1}]}\frac{a_n}{n} \right| \leq \frac{1}{N_j} \leq \left(\frac{2}{\delta_0}\right) \frac{1}{N_j}.  \end{equation}   Let $M:=N_{I}$ and $N>M$. There exists a unique $j \geq I$ such that $N_j < N \leq N_{j+1}.$ When going from $j$ to $j+1$ in \eqref{boundN_i}, we can ensure that at each step the sum is decreasing. Hence, the result follows.


We now prove $(2)$. Let $0<\varepsilon<-\frac{\log \left (1-\delta_0/2 \right)}{2}$. Repeating the same argument as in the first case, we can find $M$ such that 
\[
\left|\sum_{n \in \cA \cap[M]} \frac{a_n}{n} \right|\leq \epsilon.
\]

Let $P_0 := M$, by the definition of limit superior, one can find an infinite sequence $\{P_i\}_{i=1}^{\infty}$ such that $P_i \geq \lceil \frac{2}{\delta_0} P_{i-1} \rceil$ and $|\cA \cap [P_i]| \geq \delta_0 P_i$ for all $i \geq 1$. Thus, we have $$\vert \cA \cap (\frac{\delta_0}{2} P_i,P_i) \vert \geq \frac{\delta_0}{2} P_i.$$ 
Starting from $P_0$ and applying the greedy algorithm, we can assume that
\[
\left|\sum_{ 1 \leq n \leq \delta_0 P_i/2 \atop n \in \cA} \frac{a_n}{n} \right|\leq \epsilon.
\]
Note that 
\[
\sum_{\delta_0 P_i/2 <n \leq P_i \atop n \in \cA }\frac{1}{n}  \geq \sum_{(1- \delta_0/2)P_i < n \leq P_i} \frac{1}{n} = \log \left (\frac{1}{1-\delta_0/2} \right)>\epsilon.
\] By Lemma \ref{l:flip}, we can find $\{a_n\} \in \{-1,1\}$ such that 
\[
\left|\sum_{n \in \cA \cap [P_i]} \frac{a_n}{n} \right|\leq \frac{2}{\delta_0} \frac{1}{P_i}.
\]

\end{proof}

It is time to prove Theorems \ref{t:postive_density_harmonic_series} and \ref{t:postive_upp_density_harmonic_series}

\begin{proof}[Proofs of Theorem \ref{t:postive_density_harmonic_series} and \ref{t:postive_upp_density_harmonic_series}]
By Lemma \ref{l:deter_low_upp_large_bdd}, there exists a large integer $M$ such that for $N>M$, 
\begin{equation}\label{small_and_density}
\sum_{n \in \cA_0 \cap[N]}\frac{a_n}{n} \ll \frac{1}{N}
\end{equation}  and $|\cA_0 \cap [N]| \geq \delta N.$
Let $N_0=\lceil 10\frac{M}{\delta} \rceil,$ for $N>N_0$ we have $$\vert \cA_0 \cap (\frac{\delta}{2} N,N) \vert \geq \frac{\delta}{2} N. $$ 
We can now apply Lemma \ref{l:key_local_lemma} in the range $[\frac{\delta}{2} N, N]$ for $N>N_0$. Let $\cA= \cA_0 \cap [\frac{\delta}{2} N, N]$ and  $0<\epsilon_0<1/3$. By Lemma \ref{l:multi_str_positive_density_set}, we can choose $\epsilon_1>0$ such that 
\[
|R_{\epsilon_1}(\cA)| > N^{1-\epsilon_0}.
\] Note that by definition we have $\Omega(R_{\epsilon_1}(\cA)) \leq \frac{1}{\epsilon_1} \ll 1$.
We set
\[
\cB = \{ r s_{r} \in \cA: r \in R_{\epsilon_1}(\cA)\text{ and }s_r= \min_{rs \in \cA}\{s \} \} \quad \text{and} \quad R(\cB)=R_{\epsilon_1}(\cA)
\]  and 
$$ x_0 = -\sum_{n \in \cA_0 \cap [\frac{\delta}{2}N]} \frac{a_n}{n}.$$ Theorem \ref{t:postive_density_harmonic_series} follows by applying Lemma \ref{l:key_local_lemma}. \\

The proof of Theorem \ref{t:postive_upp_density_harmonic_series} is essentially the same as the proof of Theorem \ref{t:postive_density_harmonic_series}. The only difference is that \eqref{small_and_density} holds along the sequence $(P_i)_{i=0}^{+\infty}$ coming from the second part of Lemma \ref{l:deter_low_upp_large_bdd}. 


\end{proof}
    
    \section*{Acknowledgement} The authors would like to thank Sacha Mangerel for his insightful and valuable comments.    
    O. K. and Y-C. S. gratefully acknowledge the hospitality of Universit\'{e} Jean Monnet during their visit, when part of this work was completed. O.K. would like to thank CRM (Montreal) for providing excellent working conditions.
    During the preparation of this work, M. M. was supported by AAP Recherche 2025 UJM ``Comportements al{\'e}atoires en arithm{\'e}tique". Moreover, this research was funded in whole or in part by the French National Research Agency (ANR) under project number ``ANR-25-CE40-1961-01''.

\bibliography{harmonic_sum}
\bibliographystyle{plain}
 \end{document}